\documentclass[10pt, reqno]{amsart}

\usepackage{amsmath, amsthm, amssymb}
\usepackage[colorlinks=true,urlcolor=blue,linkcolor=blue,citecolor=blue]{hyperref}
\usepackage{enumerate}
\usepackage{geometry}
\usepackage{graphicx}
\usepackage{cleveref}
\usepackage{mathdots}
\usepackage{tikz-cd}
\usepackage{comment}

\newtheorem{theorem}{Theorem}[section]

\newtheorem{proposition}[theorem]{Proposition}

\newtheorem{definition}[theorem]{Definition}

\newtheorem{remark}[theorem]{Remark}
\newtheorem{corollary}[theorem]{Corollary}

\newcommand{\Z}{\mathbb{Z}}

\newcommand{\C}{\mathbb{C}}

\newcommand{\G}{\mathrm{GL}}
\newcommand{\localField}{K}
\newcommand{\finiteField}{k}

\newcommand{\tr}[1]{\mathrm{Tr}_{\finiteField_{#1}/\finiteField}}

\title{Epsilon Factors of Representations of Finite General Linear Groups}
\author{Rongqing Ye}
\address{Department of Mathematics, Purdue University, West Lafayette, IN 47907, USA}
\email{ye271@purdue.edu}

\author{Elad Zelingher}
\address{Department of Mathematics, Yale University, New Haven, CT 06510, USA}
\email{elad.zelingher@yale.edu}

\date{March 25, 2019. Last updated: November 5, 2020}

\begin{document}

\begin{abstract}
We define epsilon factors for irreducible representations of finite general linear groups using Macdonald's correspondence. These epsilon factors satisfy multiplicativity, and are expressible as products of Gauss sums. The tensor product epsilon factors are related to the Rankin-Selberg gamma factors, by which we prove that the Rankin-Selberg gamma factors can be written as products of Gauss sums. The exterior square epsilon factors relate the Jacquet-Shalika exterior square gamma factors and the Langlands-Shahidi exterior square gamma factors for level zero supercuspidal representations. We prove that these exterior square factors coincide in a special case.
\end{abstract}

\maketitle

\section{Introduction}\label{sec:introduction}

Long before the local Langlands correspondence was established by Harris-Taylor \cite{HarrisTaylor01} and Henniart \cite{Henniart00}, Macdonald had already established a correspondence between irreducible representations of $\G_n(\finiteField)$, $k$ a finite field, and inertia equivalence classes of admissible tamely ramified $n$-dimensional Weil-Deligne representations of $W_\localField$, where $\localField$ is a non-archimedean local field with residue field $\finiteField$ and $W_\localField$ is the Weil group of $\localField$. This correspondence matches epsilon factors. In view of \cite[(A.1)]{SilbergerZink08}, Macdonald's correspondence is the restriction of the local Langlands correspondence to level zero representations. 

Let $\mathfrak{o}$ be the ring of integers of $\localField$, $\mathfrak{p} \subset \mathfrak{o}$ be its prime ideal. Let $q$ be the size of $\finiteField$. Let $\psi$ be an additive character of $\localField$ of conductor $\mathfrak{p}$, i.e., $\psi$ is trivial on $\mathfrak{p}$ but not on $\mathfrak{o}$. Thus, it descends to a non-trivial additive character $\psi$ on $\finiteField$. Let $dx$ be the Haar measure on $\localField$ normalized such that $\mathfrak{p}$ has volume $q^{-\frac{1}{2}}$. If $\phi$ is a tamely ramified $n$-dimensional Weil-Deligne representation of $W_\localField$ corresponding to the irreducible representation $\pi$ of $\G_n(\finiteField)$, then the match of epsilon factors asserts that
$$\epsilon(\pi, \psi) = \epsilon_0(\phi, \psi, dx).$$
Here $\epsilon(\pi, \psi)$ is the Godement-Jacquet epsilon factor of $\pi$ defined in \cite{Macdonald80}, and $\epsilon_0(\phi, \psi, dx)$ is the arithmetic epsilon factor defined by Deligne in \cite{Deligne73}. We note that operations such as direct sums, tensor products, exterior powers and symmetric powers preserve tame ramification of Weil-Deligne representations of $W_\localField$. Thus in the spirit of Macdonald's correspondence, we define various $\epsilon_0$-factors of representations over the finite field with respect to these operations in a way that they match the arithmetic $\epsilon_0$-factors of Deligne, see \Cref{defn:epsilon}. The immediate benefit of such definition is that $\epsilon_0$-factors over finite fields inherit good properties from the arithmetic $\epsilon_0$-factors of Deligne. The most important ones are multiplicativity and being expressible as products of Gauss sums.

$\epsilon_0$-factors over finite fields agree with those gamma factors coming from integral representations. In \Cref{sec:tensor_epsilon_factors}, we write down explicitly the formulas for the finite $\epsilon_0$-factors with respect to the tensor product operation in terms of Gauss sums. Then we show that they are equal to the corresponding Rankin-Selberg gamma factors defined in \cite{Roditty10,Nien14,Ye18}, up to some effective constants. Since $\epsilon_0$-factors can be written as products of Gauss sums, we then prove in \Cref{cor:rankin_selberg_factor_as_a_product_of_gauss_sums} that the Rankin-Selberg gamma factors are also products of Gauss sums, which answers \cite[Conjecture 2.2]{NienZhang18}.

\Cref{sec:exterior_square_epsilon_factors} is similar to \Cref{sec:tensor_epsilon_factors}, only that we analyze the exterior square epsilon factors in this section. The exterior square epsilon factors are also multiplicative and are products of Gauss sums. We will not repeat the proofs of these statements, since the techniques are demonstrated in \Cref{sec:tensor_epsilon_factors}. The exterior square epsilon factor is equal, up to some constant $c_f$, to the exterior square gamma factor defined in \cite{YeZeligher18}. Unfortunately, the constants $c_f$ are not effective. We conjecture that $c_f = 1$, which is roughly equivalent to the statement that the two exterior square epsilon factors coming from the Jacquet-Shalika integral representation and the Langlands-Shahidi method coincide.

\emph{Acknowledgments.} We would like to thank Colin Bushnell for pointing us to \cite{Macdonald80}. We are especially grateful to James Cogdell for careful readings of the paper and his helpful comments.

\section{Tamely ramified representations}\label{sec:tamely_ramified_representations}

Let $\localField$ be a non-archimedean local field with residue field $\finiteField$ of size $q$. Let $\mathfrak{o}$ be the ring of integers of $\localField$ and $\mathfrak{p} = (\varpi)$ be the maximal ideal of $\mathfrak{o}$, where $\varpi$ is a fixed uniformizer. Then $\finiteField$ is isomorphic to $\mathfrak{o} / \mathfrak{p}$. Set $\overline{\localField}$ and $\overline{\finiteField}$ as the separable closures of $\localField$ and $\finiteField$ respectively. We have a short exact sequence
\begin{equation*}
    1 \to I
    \to \mathrm{Gal}(\overline{\localField}/\localField)
    \to \mathrm{Gal}(\overline{\finiteField}/\finiteField) \cong \widehat{\Z}
    \to 1,
\end{equation*}
where $I$ is called the inertia subgroup, and $\widehat{\Z} = \underset{m}{\varprojlim} \Z/m\Z$ is the inverse limit of the Galois groups of the finite degree field extensions of $\finiteField$. The pro-$p$ subgroup $P$ of $I$ (where $p$ is the characteristic of $\finiteField$) is the wild inertia subgroup of $\localField$. Taking the preimage of $\Z \subset \widehat{\Z}$, we obtain another short exact sequence involving the Weil group $W$ of $\localField$:
\begin{equation*}
    1 \to I \to W(\overline{\localField}/\localField) \to \Z \to 1.
\end{equation*}
For convenience, we set $W = W(\overline{\localField}/\localField)$. Let $\localField^{un} \subset \overline{\localField}$ be the maximal unramified extension of $\localField$. Let $F \in \mathrm{Gal}(\localField^{un}/\localField)$ be a geometric Frobenius element, i.e., it is the inverse image of an automorphism of $\overline{\finiteField} \slash \finiteField$, also denoted by $F$, defined by $F(x^q) = x$ for $x \in \overline{\finiteField}$. Then $W = I \rtimes \langle F \rangle$. We define a norm $\lVert \cdot \rVert$ on $W$ by setting
$$\lVert i \rVert = 1, \lVert F \rVert = q^{-1},$$
where $i \in I$.

A Weil-Deligne representation of $W$ is a pair $\phi = (\rho, N)$ satisfying the following conditions:
\begin{enumerate}[(1)]
    \item $\rho: W \to \G(V)$ is a finite dimensional representation on $V$ over $\C$, such that $\rho(w)$ is semisimple for $w \in W$, and $\ker(\rho)$ contains an open subgroup of $I$;

    \item $N: V \to V$ is nilpotent, and $\rho(w)N\rho(w)^{-1} = \lVert w \rVert \cdot N$ for $w \in W$.
\end{enumerate}
The \emph{degree} (or \emph{dimension}) of $\phi$ is set to be $\dim \rho$. If $\ker(\rho)$ contains $I$ (resp. $P$), then $\rho$ and $\phi$ are said to be \emph{unramified} (resp. \emph{tamely ramified}). Two Weil-Deligne representations $\phi = (\rho, N)$ and $\phi^\prime = (\rho^\prime, N^\prime)$ are \emph{equivalent} if there exists a linear isomorphism $\alpha: V \to V^\prime$ such that the following two diagrams commute for all $w \in W$, where $V$ and $V^\prime$ are the underlying vector spaces of $\rho$ and $\rho^\prime$ respectively.

\begin{center}
\begin{tabular}{c c c}
    \begin{tikzcd}
        V \arrow[r, "\alpha"] \arrow[d, "\rho(w)" left] & V^\prime \arrow[d, "\rho^\prime(w)"] \\
        V \arrow[r, "\alpha"] & V^\prime
    \end{tikzcd}
    & \hspace{1cm} &
    \begin{tikzcd}
        V \arrow[r, "\alpha"] \arrow[d, "N" left] & V^\prime \arrow[d, "N^\prime"] \\
        V \arrow[r, "\alpha"] & V^\prime
    \end{tikzcd}
\end{tabular}
\end{center}
Similarly, $\phi$ and $\phi^\prime$ are said to be \emph{$I$-equivalent} if the above diagrams commute with $\rho$ and $\rho^\prime$ replaced by their restrictions to $I$ with some linear isomorphism $\alpha: V \to V^\prime$. Following the notations in \cite[Section 3]{Macdonald80}, we set $\Phi^t(\G_n)$ to be the set of equivalence classes of tamely ramified Weil-Deligne representations of $W$ of degree $n$, and set $\Phi^t_I(\G_n)$ to be the set of $I$-equivalence classes of tamely ramified Weil-Deligne representations of $W$ of degree $n$.

For a Weil representation $\rho: W \to \G(V)$, Deligne \cite[Section 4 and 5]{Deligne73} defined the epsilon factors $\epsilon(\rho, \psi, dx)$ and $\epsilon_0(\rho, \psi, dx)$ associated to it, where $\psi$ is a non-trivial additive character of $\localField$ and $dx$ is an arbitrary Haar measure on $\localField$. These two epsilon factors are non-zero constants related by
\begin{equation}\label{eqn:relation_of_epsilon_factors_for_weil_representation}
    \epsilon(\rho, \psi, dx) = \epsilon_0(\rho, \psi, dx) \det \left(-F, V^I\right)^{-1},
\end{equation}
where $V^I$ is the maximal subspace of $V$ on which $\rho(I)$ acts trivially. For a Weil-Deligne representation $\phi = (\rho, N)$, following \cite[8.12]{Deligne73}, we can also define the epsilon factors associated to it by
\begin{equation}\label{eqn:epsilon0_factor_of_weil_deligne_representation}
    \epsilon_0(\phi, \psi, dx) = \epsilon_0(\rho, \psi, dx),
\end{equation}
and
\begin{equation}\label{eqn:epsilon_factor_of_weil_deligne_representation}
    \epsilon(\phi, \psi, dx) = \epsilon_0(\rho, \psi, dx)\det\left(-F, V_N^I\right)^{-1},
\end{equation}
where $V^I_N$ is the null space of $N: V^I \to V^I$ ($N$ preserves $V^I$ because $\rho(i)N\rho(i)^{-1} = \lVert i \rVert \cdot N = N$ for $i \in I$).

These epsilon factors are in general hard to made explicit. Deligne computed $\epsilon_0$ for the case where $\rho$ is tamely ramified.

\begin{theorem}[{\cite[Section 5.16]{Deligne73}}]\label{thm:explicit_epsilon0_factor_for_tame_representation}
    Let $\rho: W = W(\overline{\localField}/\localField) \to \G(V)$ be a tamely ramified representation, i.e., $\rho(P) \equiv 1$. Then
    \begin{enumerate}[(1)]
        \item $\rho$ is a direct sum of induced representations of tamely ramified characters $\chi_i : K_i^{\times} \rightarrow \mathbb{C}^{\times}$ of unramified extensions $\localField_i/\localField$, i.e.,
        $$\rho = \sum_{i=1}^r \mathrm{Ind}_{W(\overline{\localField}/\localField_i)}^W \chi_i,$$
        where $i = 1, \cdots, r$ for some integer $r$, and $\chi_i$ is treated as a character of $W(\overline{\localField}/\localField_i)$ via the natural map $W(\overline{\localField}/\localField_i) \to W(\overline{\localField}/\localField_i)^{ab} \cong \localField_i^\times$ from local class field theory, normalized so that $F$ is sent to $\varpi$.

        \item Let $\psi$ be an additive character of $\localField$ of conductor $\mathfrak{p}$, that is, $\psi$ is trivial on $\mathfrak{p}$ but not on $\mathfrak{o}$. Then $\psi$ can be treated as a character $\psi_\finiteField$ of $\finiteField$ via the isomorphism $\mathfrak{o} / \mathfrak{p} \to \finiteField$. Let $dx$ be the Haar measure on $\localField$, normalized such that $\int_\mathfrak{p} \, dx = q^{-\frac{1}{2}}$. Then
        $$\epsilon_0(\rho, \psi, dx) = (-1)^{\dim V} q^{-\frac{\dim V}{2}} \prod_{i=1}^r \tau(\widetilde{\chi}_i, \psi_{\finiteField_i}),$$
        where $\finiteField_i$ is the residue field of $\localField_i$, $\psi_{\finiteField_i} = \psi_\finiteField \circ \tr{i}$ is an additive character of $\finiteField_i$, $\widetilde{\chi}_i$ is the multiplicative character of $\finiteField_i^\times$ defined by $\chi_i$, and $\tau(\widetilde{\chi}_i, \psi_{\finiteField_i})$ is the Gauss sum of $\widetilde{\chi}_i$ with respect to the additive character $\psi_{\finiteField_i}$ defined by
        $$\tau(\widetilde{\chi}_i, \psi_{\finiteField_i}) = - \sum_{x \in \finiteField_i^\times} \widetilde{\chi}_i \left(x^{-1}\right) \psi_{\finiteField_i}(x).$$
    \end{enumerate}
\end{theorem}

These epsilon factors are equal for two representations in the same equivalence class. They might not be equal for two representations that are only in the same $I$-equivalence class. However, an immediate corollary of the above theorem is that $\epsilon_0(\phi, \psi, dx)$ is well-defined on $\Phi^t_I(\G_n)$, since the Frobenius element $F$ is not involved in the explicit formula of $\epsilon_0(\rho, \psi, dx)$.

\begin{corollary}\label{cor:epsilon0_factors_are_defined_on_I_equivence_classes}
    Let $\phi = (\rho, N)$ and $\phi^\prime = (\rho^\prime, N^\prime)$ be two tamely ramified Weil-Deligne representations of $W$ such that $\rho_I \cong \rho^\prime_I$. Then for any Haar measure $dx$ on $\localField$,
    $$\epsilon_0(\phi, \psi, dx) = \epsilon_0(\phi^\prime, \psi, dx).$$
    In particular, the above equality holds if $\phi$ and $\phi^\prime$ are $I$-equivalent.
\end{corollary}

Tamely ramified Weil-Deligne representations are parameterized in \cite[Section 3]{Macdonald80}. We now recall this parameterization. A \emph{partition} of a non-negative integer $n$ is a tuple $\lambda = (\lambda_1, \lambda_2, \cdots, \lambda_r)$ such that $\lambda_1 \ge \lambda_2 \ge \cdots \ge \lambda_r > 0$ and $|\lambda| := \sum_{i=1}^r \lambda_i = n$. We define $\mathfrak{n}(\lambda) := r$ to be the number of parts of $\lambda$. For convenience, we sometimes write $\lambda = (\lambda_i)$ and ignore the ordering of $\lambda_i$'s. For example, if $\lambda = (\lambda_i)$ and $\mu = (\mu_j)$ are partitions of $n$ and $m$ respectively, then $\lambda + \mu := (\lambda_i; \mu_j)$, the concatenation of these two partitions, is a partition of $n + m$ and $\lambda \cdot \mu := (\lambda_i\mu_j)$ is a partition of $nm$. We denote the \emph{empty partition} of $0$ by $()$. We set $\mathcal{P}_n$ to be the set of partitions of $n$ and $\mathcal{P} = \bigcup_{n \ge 0} \mathcal{P}_n$.

Let $\finiteField_n$ be the (unique) field extension of $\finiteField$ of degree $n$ in $\overline{\finiteField}$. Let $\Gamma_n$ be the character group of $\finiteField_n^\times$. For $m | n$, the norm map $N_{n, m}: \finiteField_n^\times \to \finiteField_m^\times$ induces an embedding $N_{n, m}: \Gamma_m \to \Gamma_n$. $\{\Gamma_n\}$ forms a directed system under these norm maps, and we define
$$\Gamma = \varinjlim \Gamma_n.$$
The Frobenius element $F$ acts on $\Gamma$ by $F \gamma = \gamma^q$ for $\gamma \in \Gamma$. We identify $\Gamma_n$ with the subgroup $\{\gamma \in \Gamma: F^n \gamma = \gamma\}$ of $\Gamma$. If $f$ is an $F$-orbit in $\Gamma$, we define the \emph{degree} $d(f)$ of $f$ to be the cardinality of $f$. Thus if $\gamma \in f$, then $\gamma \in \Gamma_{d(f)}$.

Let $P_n(\Gamma)$ be the set of partition-valued functions $\lambda: \Gamma \to \mathcal{P}$ such that
\begin{enumerate}[(i)]
\item $\lambda \circ F = \lambda$, i.e., $\lambda$ is constant on the $F$-orbits;

\item $\sum_{\gamma \in \Gamma} |\lambda(\gamma)| = n$ (Since $|\lambda(\gamma)| \ge 0$, it is actually a finite sum).
\end{enumerate}

If $\lambda \in P_n(\Gamma)$, then since $\lambda \circ F = \lambda$, it makes sense to define $\lambda(f) := \lambda(\gamma)$ for any $F$-orbit $f$ and any $\gamma \in f$. Let $\gamma \in \Gamma_n$ be a multiplicative character of $\finiteField_n^\times$ and let $\psi_n$ be the additive character of $\finiteField_n$ defined by $\psi$. In \Cref{thm:explicit_epsilon0_factor_for_tame_representation}, we have defined the Gauss sum of $\gamma$ with respect to $\psi_n$:
$$\tau(\gamma, \psi_n) = - \sum_{x \in \finiteField_n^\times} \gamma \left( x^{-1} \right)\psi_n(x).$$
We observe that $\tau(F(\gamma), \psi_n) = \tau(\gamma, \psi_n)$: since $\psi_n(x^q) = \psi \circ \tr{n}(x^q) = \psi \circ \tr{n}(x)  = \psi_n(x)$ for $x \in \finiteField_n$, and $x \mapsto x^q$ is an automorphism of $\finiteField_n^\times$, we have
\begin{equation*}
    \begin{split}
        \tau(F(\gamma), \psi_n)
        &= \tau(\gamma^q, \psi_n) = - \sum_{x \in \finiteField_n^\times} \gamma\left(x^{-1}\right)^q \psi_n(x) \\
        &= - \sum_{x \in \finiteField_n^\times} \gamma\left(x^{q}\right)^{-1} \psi_n(x)
        = - \sum_{x \in \finiteField_n^\times} \gamma\left(x^{q}\right)^{-1} \psi_n(x^q) \\
        &= - \sum_{x \in \finiteField_n^\times} \gamma\left(x\right)^{-1} \psi_n(x)
        = \tau(\gamma, \psi_n).
    \end{split}
\end{equation*}
Therefore, we can define the Gauss sum of an $F$-orbit. Let $f$ be an $F$-orbit of $\gamma \in \Gamma_n$, we define
$$\tau(f, \psi_n) := \tau(\gamma, \psi_n).$$

\begin{theorem}[\cite{Macdonald80}]\label{thm:parameterization_of_tame_representations}
    There is a natural bijection from $P_n(\Gamma)$ onto $\Phi^t_I(\G_n)$.
\end{theorem}

Given $\lambda \in P_n(\Gamma)$, we use $\phi_\lambda = (\rho_\lambda, N_\lambda)$ to denote a representative in the corresponding $I$-equivalence class of tamely ramified Weil-Deligne representations. The underlying vector space of $\rho_\lambda$ is denoted by $V_\lambda$. If there is no ambiguity, we will use $\phi_\lambda$ to denote its $I$-equivalence class. If $\lambda$ is supported on a (single) degree $n$ $F$-orbit $f$, i.e., $\lambda(\gamma) = (1)$ for $\gamma \in f$ and $\lambda(\gamma) = ()$ otherwise, then we will sometimes use $\phi_f = (\rho_f, N_f)$ to mean $\phi_\lambda = (\rho_\lambda, N_\lambda)$. We will also use $V_f$ for $V_\lambda$ in this case.

With \Cref{thm:parameterization_of_tame_representations} and the notations introduced above, we can restate \Cref{thm:explicit_epsilon0_factor_for_tame_representation} as follows. See \cite[Section 3]{Macdonald80} for more details.

\begin{theorem}\label{thm:explicit_epsilon0_factor_for_tame_representation_in_macdonald_parameterization}
Let $\lambda \in P_n(\Gamma)$ and $\phi_\lambda = (\rho_\lambda, N_\lambda)$ be as above. 
    \begin{enumerate}[(1)]
        \item Let $(\rho_\lambda)_I$ be the restriction of $\rho_\lambda$ to $I$. Then
        $$(\rho_\lambda)_I = \sum_{\gamma \in \Gamma} |\lambda(\gamma)|\gamma = \sum_{f \in F\backslash \Gamma} |\lambda(f)|\sum_{\gamma \in f} \gamma,$$
        where $F \backslash \Gamma$ is the set of $F$-orbits in $\Gamma$, and the sums above are direct sums.

        \item Let $\psi$, $\psi_\finiteField$ and $dx$ be as in \Cref{thm:explicit_epsilon0_factor_for_tame_representation}. For an $F$-orbit $f$ in $\Gamma$, let $\psi_{d(f)} = \psi_\finiteField \circ \tr{d(f)}$ be the additive character of $\finiteField_{d(f)}$ defined by $\psi_\finiteField$. Then
        $$\epsilon_0(\phi_\lambda, \psi, dx) = (-1)^n q^{-\frac{n}{2}} \prod_{f} \tau(f, \psi_{d(f)})^{|\lambda(f)|}.$$
    \end{enumerate}
\end{theorem}

Let $\gamma \in \Gamma_n$ and let $f$ be the $F$-orbit of $\gamma$, i.e., $f$ is the set $\{\gamma, F\gamma, \cdots, F^{n-1}\gamma\}$ with duplicated elements removed. The size of $f$, or equivalently the degree of $f$, is not necessary $n$. Let $m = d(f)$, then $m | n$. In this case, $\gamma$ is actually an element in the subgroup $\Gamma_m \subset \Gamma_n$, where the embedding of $\Gamma_m$ into $\Gamma_n$ is given by the norm map $N_{n, m}: \finiteField_n \to \finiteField_m$. We introduce the notion of the \emph{$F$-set generated by $\gamma$ with respect to $\Gamma_n$} to be the (multi-)set $\{\gamma, F\gamma, \cdots, F^{n-1}\gamma\}$, possibly with duplicated elements. Let $h$ be the $F$-set generated by $\gamma$ with respect to $\Gamma_n$, then $h$ consists of $\frac{n}{m}$ copies of $f$. Since the elements of $h$ are in the same $F$-orbit, we can define the Gauss sum for $h$ to be
$$\tau(h, \psi_n) := \tau(\alpha, \psi_n),$$
for any $\alpha \in h$.

\begin{theorem}[Hasse-Davenport, {\cite[5.12]{Deligne73}}]\label{thm:hasse-davenport}
Let $m | n$ and let $\gamma \in \Gamma_m$ be a character of $\finiteField_m^\times$, such that its $F$-orbit $f$ is of degree $m$. Let $h$ be the $F$-set generated by $\gamma$ with respect to $\Gamma_n$. Then $\tau(h, \psi_n) = \tau(f, \psi_m)^{\frac{n}{m}}$.
\end{theorem}

Let
$$G = \mathrm{Gal}(\overline{\localField}/\localField) \supset G_0 = I \supset G_1 = P \supset G_2 \supset \cdots \supset G_N = 1$$
be the filtration of the ramification subgroups \cite{Serre79}. For a Weil-Deligne representation $\phi = (\rho, N)$, where $\rho: W \to \G(V)$ for some complex vector space $V$, we define its \emph{Artin conductor} $a(\phi)$ by
$$a(\phi) = \sum_{j = 0}^N \dim \left(V/V^{G_j}\right) \frac{|\rho(G_j)|}{|\rho(G_0)|},$$
where $V^{G_j}$ is the subspace of $\rho(G_j)$-fixed vectors in $V$. We note that by definition, $\ker \rho$ contains an open subgroup of $G_0 = I$, and since $I$ is compact, the image $\rho(G_0)$ of $G_0$ is finite. Hence so are the other $\rho(G_j)$ for $j = 1, \dots, N$, as $G_j \subset G_0$. Isolating the $j=0$ term in $a(\phi)$, we can write
$$a(\phi) = \dim\left( V/V^I \right) + b(\phi),$$
where $b(\phi)$ is called the \emph{Swan conductor} of $\phi$. Deligne defines $\epsilon(s, \phi, \psi, dx)$, which is an indispensable ingredient in the local Langlands correspondence, by the formula
$$\epsilon(s, \phi, \psi, dx) = \epsilon(\phi, \psi, dx)q^{\left(\dim V - a(\phi)\right)s}.$$
See \cite[5.5 and 8.18]{Deligne73} for details, and note that $n(\psi) = -1$ for our choice of $\psi$.

Let $\lambda \in P_n(\Gamma)$ and $\phi_\lambda = (\rho_\lambda, N_\lambda)$ be the tamely ramified Weil-Deligne representation under the bijection in \Cref{thm:parameterization_of_tame_representations}. Since $\rho_\lambda(P)$ acts trivially on $V_\lambda$, we have $\dim \left( V_\lambda / V_\lambda^P \right) = 0$, so $b(\phi_\lambda) = 0$. Thus in this case, $a(\phi) = \dim \left( V_\lambda / V_\lambda^I \right) = \dim V_\lambda - \dim V_\lambda^I$. Therefore,
\begin{equation}\label{eqn:arithmetic_epsilon_of_tame_representation}
    \epsilon(s, \phi_\lambda, \psi, dx) =\epsilon(\phi_\lambda, \psi, dx)q^{\left( \dim V_\lambda^I \right)s} = \epsilon(\phi_\lambda, \psi, dx)q^{|\lambda(1)|s},
\end{equation}
where $\lambda(1)$ is the partition of $\lambda$ evaluated at the trivial character $1$ of $\finiteField^\times$.

Finally, we want to discuss when $\epsilon(\phi_\lambda, \psi, dx) = \epsilon_0(\phi_\lambda, \psi, dx)$. From \Cref{eqn:epsilon0_factor_of_weil_deligne_representation} and \Cref{eqn:epsilon_factor_of_weil_deligne_representation}, we have
$$\epsilon(\phi_\lambda, \psi, dx) = \epsilon_0(\phi_\lambda, \psi, dx)\det\left(-F, V_{N_\lambda}^I\right)^{-1}.$$
Thus, we want to know when $\det\left(-F, (V_\lambda)_{N_\lambda}^I\right)$ equals $1$. An obvious sufficient condition is $V_\lambda^I = 0$.

\begin{proposition}\label{prop:sufficient_condition_for_I_invariant_subspace_to_vanish}
    Let $\lambda \in P_n(\Gamma)$ and let $\phi_\lambda = (\rho_\lambda, N_\lambda)$ be the corresponding $I$-equivalence class of a tamely ramified Weil-Deligne representation. Then $V_\lambda^I = 0$ if and only if $\lambda(1) = ()$.
\end{proposition}

\begin{proof}
    Since we only concentrate on the $I$-fixed subspace, it makes sense to work with $I$-equivalence classes directly. By \Cref{thm:explicit_epsilon0_factor_for_tame_representation_in_macdonald_parameterization},
    $$(\rho_\lambda)_I = \sum_{\gamma \in \Gamma} |\lambda(\gamma)|\gamma.$$
    Therefore, $V_\lambda$ has a non-trivial $I$-fixed vector if and only if $\lambda(1) \ne ()$.
\end{proof}

\begin{corollary}\label{cor:sufficient_condition_for_I_invariant_subspace_to_vanish}
    Let $\lambda \in P_n(\Gamma)$ be such that $\lambda(1) = ()$, i.e., $\lambda$ is not supported on $1$.Then
    $$\displaystyle \epsilon(s, \phi_\lambda, \psi, dx) = \epsilon(\phi_\lambda, \psi, dx) = \epsilon_0(\phi_\lambda, \psi, dx).$$
\end{corollary}

\begin{proof}
    The first equality comes from \Cref{eqn:arithmetic_epsilon_of_tame_representation} and $\lambda(1) = ()$. Since $V_\lambda^I = 0$ by \Cref{prop:sufficient_condition_for_I_invariant_subspace_to_vanish}, $(V_\lambda)_{N_\lambda}^I=0$, so $\det \left( -F, (V_\lambda)_{N_\lambda}^I \right) = 1$. Then the second equality follows.
\end{proof}

\section{Macdonald's correspondence and epsilon factors}\label{sec:macdonald_correspondence_and_epsilon_factors}

From now on, we fix an additive character $\psi$ on $\localField$ of conductor $\mathfrak{p}$, and let $dx$ be the Haar measure on $\localField$ such that $\int_\mathfrak{p} \, dx = q^{-\frac{1}{2}}$. $dx$ is self dual with respect to $\psi$. We also use $\psi$ to denote the additive character on $\finiteField$ defined by $\psi$ via the isomorphism $\finiteField \cong \mathfrak{o} / \mathfrak{p}$. Let $\finiteField_n$ be the (unique) field extension of $\finiteField$ of degree $n$. Let $\psi_n$ be the additive character on $\finiteField_n$ defined by $\psi_n = \psi \circ \tr{n}$.

Let $\Pi(\G_n(\finiteField))$  be the set of equivalence classes of irreducible representations of $\G_n(\finiteField)$. Green \cite{Green55} established a natural bijection between $\Pi(\G_n(\finiteField))$ and $P_n(\Gamma)$, see also \cite[Section 1]{Macdonald80}. If $\lambda \in P_n(\Gamma)$, we denote by $\pi_\lambda \in \Pi(\G_n(\finiteField))$ the corresponding (equivalence class of the) irreducible representation. If $\lambda$ is supported only on a (single) orbit $f$ of degree $n$, we will also write $\pi_f$ for $\pi_\lambda$.

In light of Green's correspondence and \Cref{thm:parameterization_of_tame_representations}, Macdonald obtained the following

\begin{theorem}[Macdonald's correspondence, \cite{Macdonald80}]\label{thm:macondald_correspondence}
    There is a canonical bijection
    $$\mathcal{M}: \Pi(\G_n(\finiteField)) \to \Phi^t_I(\G_n)$$
    such that $\mathcal{M}(\pi_\lambda) = \phi_\lambda$ for $\lambda \in P_n(\Gamma)$ and
    \begin{equation}\label{eqn:equality_of_epsilon_factors_under_macdonald_correspondence}
        \epsilon(\pi_\lambda, \psi) = \epsilon_0(\phi_\lambda, \psi, dx),
    \end{equation}
    where $\epsilon(\pi, \psi)$ is the Godement-Jacquet epsilon factor defined in \cite[Section 2]{Macdonald80}.
\end{theorem}

Macdonald's correspondence $\mathcal{M}$ and \Cref{eqn:equality_of_epsilon_factors_under_macdonald_correspondence} inspire the following definition of epsilon factors:

\begin{definition}\label{defn:epsilon}
    Let $\pi \in \Pi(\G_n(\finiteField))$, and let $r$ be an operation on Weil-Deligne representations of $W$ that preserves tame ramification, for example $r$ can be direct sum, tensor product, exterior power or symmetric power. We define a non-zero constant $\epsilon_0(\pi, r, \psi)$ associated to the pair $(\pi, r)$ by
    \begin{equation*}
        \epsilon_0(\pi, r, \psi) := \epsilon_0(r\left(\mathcal{M}(\pi)\right), \psi, dx),
    \end{equation*}
\end{definition}

\begin{remark}
    \begin{enumerate}[(1)]
        \item There is exactly one Haar measure $dx$ on $\localField$ such that $\int_\mathfrak{p} \, dx = q^{-\frac{1}{2}}$, but there are $q-1$ different additive characters of conductor $\mathfrak{p}$ on $\localField$. Therefore, we keep $\psi$ and omit $dx$ in the notation $\epsilon_0(\pi, r, \psi)$, though we have fixed $\psi$ and $dx$ at the beginning of the section. One should keep in mind that the definition of $\epsilon_0(\pi, r, \psi)$ depends on the choices of both $\psi$ and $dx$.

        \item If $r = id$ is the identity operation, then \Cref{eqn:equality_of_epsilon_factors_under_macdonald_correspondence} is the same as
        $$\epsilon(\pi, \psi) = \epsilon_0(\pi, id, \psi),$$
        where $\epsilon(\pi, \psi)$ is the Godement-Jacquet epsilon factor of $\pi$.

        \item If $r$ is either the direct sum or the tensor product, then it is a binary operator, so the corresponding epsilon factors should receive as input two representations. For example, when $r = \oplus$ is the direct sum, it should be understood that the epsilon factors are defined for two representations $\pi_1 \in \Pi(\G_n(\finiteField))$ and $\pi_2 \in \Pi(\G_m(\finiteField))$ and we write
        $$\epsilon_0(\pi_1 \boxplus \pi_2, \psi) := \epsilon_0(\pi_1, \pi_2, \oplus, \psi) = \epsilon_0(\mathcal{M}(\pi_1) \oplus \mathcal{M}(\pi_2), \psi).$$
        And for $r = \otimes$ the tensor product, we write
        $$\epsilon_0(\pi_1 \times \pi_2, \psi) := \epsilon_0(\pi_1, \pi_2, \otimes, \psi) = \epsilon_0(\mathcal{M}(\pi_1) \otimes \mathcal{M}(\pi_2), \psi).$$

        \item If $r$ is one of these four operations listed in \Cref{defn:epsilon}, then it preserves tame ramification. That is to say, if $\phi$ is tamely ramified Weil-Deligne representation of $W$ (or a pair of tamely ramified Weil-Deligne representations when $r$ is the direct sum or the tensor product), then so is $r(\phi)$.
    \end{enumerate}
\end{remark}

It follows immediately from the definition of $\epsilon_0(\pi, r, \psi)$ that epsilon factors of irreducible representations of general linear groups over $\finiteField$ enjoy the same good properties as those of epsilon factors of tamely ramified Weil-Deligne representations over $\localField$. For example, they satisfy multiplicativity \cite[(5.2)]{Deligne73}, and they can be written as products of Gauss sums by \Cref{thm:explicit_epsilon0_factor_for_tame_representation} or \Cref{thm:explicit_epsilon0_factor_for_tame_representation_in_macdonald_parameterization}. To illustrate these properties, we will look at the case where $r$ is the identity operator for the rest of this section. For simplicity of notations, we define
$$\epsilon_0(\pi, \psi) := \epsilon_0(\pi, id, \psi) = \epsilon_0(\mathcal{M}(\pi), \psi).$$

Let $\lambda \in P_n(\Gamma)$ and $\mu \in P_m(\Gamma)$. Then by multiplicativity of $\epsilon_0$-factors of Weil-Deligne representations,
\begin{equation*}
    \epsilon_0(\pi_\lambda \boxplus \pi_\mu, \psi)
    = \epsilon_0(\phi_\lambda \oplus \phi_\mu, \psi)
    = \epsilon_0(\phi_\lambda, \psi)\epsilon_0(\phi_\mu, \psi)
    = \epsilon_0(\pi_\lambda, \psi)\epsilon_0(\pi_\mu, \psi).
\end{equation*}
This is the multiplicativity property for $\epsilon_0$-factors of irreducible representations of general linear groups over $\finiteField$. On the other hand, from \Cref{thm:explicit_epsilon0_factor_for_tame_representation_in_macdonald_parameterization} and \Cref{defn:epsilon}, we can express $\epsilon_0(\pi_\lambda, \psi)$ as a product of Gauss sums:
\begin{equation*}
    \epsilon_0(\pi_\lambda, \psi) = \epsilon_0(\phi_\lambda, \psi) = (-1)^n q^{-\frac{n}{2}} \prod_{f} \tau(f, \psi_{d(f)})^{|\lambda(f)|}.
\end{equation*}

There are two ways to express $\epsilon_0(\pi_\lambda \boxplus \pi_\mu, \psi)$ as a product of Gauss sums. The first way is to use the multiplicativity of $\epsilon_0$-factors:
\begin{equation*}
    \begin{split}
        \epsilon_0(\pi_\lambda \boxplus \pi_\mu, \psi)
        &= \epsilon_0(\pi_\lambda, \psi)\epsilon_0(\pi_\mu, \psi) \\
        &= (-1)^{n+m} q^{-\frac{n+m}{2}} \prod_f \tau(f, \psi_{d(f)})^{|\lambda(f)|+|\mu(f)|}.
    \end{split}
\end{equation*}
On the other hand, by the explicit correspondence between $P_n(\Gamma)$ and $\Phi^t_I(\G_n)$ as in \cite[Section 3]{Macdonald80}, we see that $\phi_\lambda \oplus \phi_\mu$ is a tamely ramified representation of degree $n+m$ corresponding to $\nu = \lambda + \mu \in P_{n+m}(\Gamma)$. Here, we have for $\gamma \in \Gamma$,
$$\nu(\gamma) = \lambda(\gamma) + \mu(\gamma),$$
where addition above is the concatenation of two partitions. Thus,
\begin{equation*}
    \epsilon_0(\pi_\lambda \boxplus \pi_\mu, \psi)
    = \epsilon_0(\phi_\nu, \psi)
    = (-1)^{n+m} q^{-\frac{n+m}{2}} \prod_f \tau(f, \psi_{d(f)})^{|\nu(f)|}.
\end{equation*}
These two ways agree because $|\nu(f)| = |\lambda(f)| + |\mu(f)|$. Moreover, since $\phi_\nu = \phi_\lambda \oplus \phi_\mu$, we can define the addition $\pi_\lambda \boxplus \pi_\mu$ to be $\pi_\nu$.

\section{Tensor product epsilon factors}\label{sec:tensor_epsilon_factors}

This section is devoted to the tensor product epsilon factors $\epsilon_0(\pi_\lambda \times \pi_\mu, \psi)$. First of all, we show that they are equal to products of Gauss sums. Secondly, we relate them to the Rankin-Selberg gamma factors defined in \cite{Roditty10,Nien14,Ye18}. Their relation to the Rankin-Selberg gamma factors can be used to solve a conjecture made by Nien and Zhang \cite[Conjecture 2.2]{NienZhang18}.

\subsection{Gauss sums}
In terms of Green's classification, irreducible \emph{cuspidal} representations of $\G_n(\finiteField)$ are in a one-to-one correspondence with $F$-orbits in $\Gamma$ of degree $n$. In other words, if $\pi$ is an irreducible cuspidal representation of $\G_n(\finiteField)$, then it corresponds to $\lambda \in P_n(\Gamma)$ supported on a (single) degree $n$ $F$-orbit $f$, i.e., $\lambda(\gamma) = (1)$ if $\gamma \in f$ and $\lambda(\gamma) = ()$ otherwise. In this case, we will use $\pi_f$ instead of $\pi_\lambda$, and we will define $\phi_f$ and $\rho_f$ by $\mathcal{M}(\pi_f) = \phi_f = (\rho_f, 0)$. Note that the nilpotent map corresponding to $\pi_f$ is $0$, since $\pi_f$ is cuspidal, and the corresponding $\phi_f$ is irreducible.

\begin{theorem}\label{thm:multiplicativity_of_tensor_epsilon_factor}
    Let $\pi_\lambda$ and $\pi_\mu$ be two irreducible representations parameterized by $\lambda \in P_n(\Gamma)$ and $\mu \in P_m(\Gamma)$ respectively. Then
    $$\epsilon_0(\pi_\lambda \times \pi_\mu, \psi) = \prod_{f, g} \epsilon_0(\pi_f \times \pi_g, \psi)^{|\lambda(f)||\mu(g)|},$$
    where $f$ and $g$ run over all the $F$-orbits in $\Gamma$.
\end{theorem}

\begin{proof}
    Let $\phi_\lambda = (\rho_\lambda, N_\lambda) = \mathcal{M}(\pi_\lambda)$ and $\phi_\mu = (\rho_\mu, N_\mu) = \mathcal{M}(\pi_\mu)$. Then by \Cref{thm:explicit_epsilon0_factor_for_tame_representation_in_macdonald_parameterization},
    \begin{equation*}
        (\rho_\lambda)_I = \sum_{\gamma \in \Gamma} |\lambda(\gamma)|\gamma = \sum_{f \in F \backslash \Gamma} |\lambda(f)| \sum_{\gamma \in f} \gamma,
    \end{equation*}
    and
    \begin{equation*}
        (\rho_\mu)_I = \sum_{\gamma \in \Gamma} |\mu(\gamma)|\gamma = \sum_{g \in F\backslash \Gamma} |\mu(g)| \sum_{\gamma \in g}\gamma.
    \end{equation*}
    Let $\phi = (\rho, N) = \phi_\lambda \otimes \phi_\mu$. Then $\rho = \rho_\lambda \otimes \rho_\mu$, and
    $$\rho_I = (\rho_\lambda)_I \otimes (\rho_\mu)_I = \sum_{f, g}|\lambda(f)||\mu(g)| \sum_{\gamma_1 \in f, \gamma_2 \in g} \gamma_1\gamma_2 = \sum_{f, g}|\lambda(f)||\mu(g)| (\rho_f \otimes \rho_g)_I.$$
    Here $(\rho_f, 0)$ are $(\rho_g, 0)$ correspond to the cuspidal representations $\pi_f$ and $\pi_g$ respectively. The last equality above comes from
    $$(\rho_f \otimes \rho_g)_I = (\rho_f)_I \otimes (\rho_g)_I = \sum_{\gamma_1 \in f}\gamma_1 \cdot \sum_{\gamma_2 \in g}\gamma_2 = \sum_{\gamma_1 \in f, \gamma_2 \in g}\gamma_1\gamma_2.$$
    Therefore, we have by \Cref{cor:epsilon0_factors_are_defined_on_I_equivence_classes} and multiplicativity of $\epsilon_0$ \cite[(5.2)]{Deligne73} (see also the example at the end of \Cref{sec:macdonald_correspondence_and_epsilon_factors}),
    $$\epsilon_0(\phi, \psi) = \prod_{f, g} \epsilon_0(\rho_f \otimes \rho_g, \psi)^{|\lambda(f)||\mu(g)|}.$$
    Now by \Cref{defn:epsilon},
    $$\epsilon_0(\pi_\lambda \times \pi_\mu, \psi) = \epsilon_0(\phi, \psi) = \prod_{f, g} \epsilon_0(\rho_f \otimes \rho_g, \psi)^{|\lambda(f)||\mu(g)|} = \prod_{f, g} \epsilon_0(\pi_f \times \pi_g, \psi)^{|\lambda(f)||\mu(g)|}.$$
\end{proof}

With \Cref{thm:multiplicativity_of_tensor_epsilon_factor}, we can reduce the study of $\epsilon_0(\pi_\lambda \times \pi_\mu, \psi)$ to that of $\epsilon_0(\pi_f \times \pi_g, \psi)$ for $F$-orbits $f$ and $g$ in $\Gamma$. Let $n \ge m$ be two positive integers. Let $f$ be an $F$-orbit of degree $n$. Fix an $\alpha \in f$, then $\alpha \in \Gamma_n$ and $f = \{\alpha^{q^i}: i = 0, 1, \dots, n-1\}$. Similarly, let $g = \{\beta^{q^j}: j = 0, 1, \dots, m-1\}$ be an $F$-orbit of degree $m$ for some $\beta \in \Gamma_m$. Let $\pi_f$ and $\pi_g$ be the corresponding cuspidal representations of $\G_n(\finiteField)$ and $\G_m(\finiteField)$ respectively. With these notations, we have the following

\begin{theorem}\label{thm:tensor_epsilon_factor_as_a_product_of_gauss_sums}
    Let $d = \gcd(n, m)$ be the greatest common divisor of $n$ and $m$, and let $l = \frac{nm}{d}$ be their least common multiple. Then
    \begin{equation*}
        \epsilon_0(\pi_f \times \pi_g, \psi) = (-1)^{nm}q^{-\frac{nm}{2}}\prod_{i=0}^{d-1} \tau(\alpha\beta^{q^i}, \psi_l),
    \end{equation*}
    where $\alpha\beta^{q^i} \in \Gamma_l$ is defined by
    $$(\alpha\beta^{q^i})(x) = \alpha\left(N_{l, n}(x)\right)\beta\left(\left(N_{l, m}(x)\right)^{q^i}\right),$$
    for $x \in \finiteField_l^\times$.
\end{theorem}

\begin{proof}
    Let $\phi_f = (\rho_f, N_f) = \mathcal{M}(\pi_f)$ and $\phi_g = (\rho_g, N_g) = \mathcal{M}(\pi_g)$. By \Cref{thm:explicit_epsilon0_factor_for_tame_representation_in_macdonald_parameterization}, we have
    \begin{equation*}
        (\rho_f)_I = \sum_{i=0}^{n-1}\alpha^{q^i} \text{ and } (\rho_g)_I = \sum_{j=0}^{m-1}\beta^{q^j},
    \end{equation*}
    where $(\rho_f)_I, (\rho_g)_I$ are the restrictions to the inertia subgroup $I$ of $\rho_f$, $\rho_g$ respectively. If we set $\phi = (\rho, N) = \phi_f \otimes \phi_g \in \Phi^t_I(\G_{nm})$, then
    \begin{equation*}
        \rho_I = (\rho_f \otimes \rho_g)_I = (\rho_f)_I \otimes (\rho_g)_I
               = \left(\sum_{i=0}^{n-1}\alpha^{q^i}\right) \otimes \left(\sum_{j=0}^{m-1}\beta^{q^j}\right)
               = \sum_{i, j} \alpha^{q^i}\beta^{q^j}.
    \end{equation*}
    We claim that these $\alpha^{q^i}\beta^{q^j}$'s are elements of $\Gamma_l$, and that they are partitioned into the $d$ $F$-sets (might not be distinct) $h_0, h_1, \dots, h_{d-1}$, where $h_i$ is the $F$-set generated by $\alpha\beta^{q^i}$ with respect to $\Gamma_l$. On one hand, each $h_i$ is of size $l$, so there are $ld = nm$ characters in $h_0, h_1, \dots, h_{d-1}$ in total. On the other hand, $\alpha^{q^i}\beta^{q^j} \in h_{\delta}$ if and only if there exists an integer $t$ such that $\alpha^{q^i}\beta^{q^j} = \alpha^{q^t}\beta^{q^{\delta+t}}$. Since $\alpha^{q^n} = \alpha$ and $\beta^{q^m} = \beta$, $\alpha^{q^i}\beta^{q^j} \in h_{\delta}$ if and only if the following system of congruence equations has a solution:
    \begin{equation*}
        \left\{
        \begin{split}
            &t \equiv i \mod{n} \\
            &t \equiv j - \delta \mod{m}
        \end{split}
        \right..
    \end{equation*}
    This system has a solution if and only if $i \equiv j - \delta \mod{d}$, or equivalently $\delta \equiv j - i \mod{d}$. Therefore, for each $\alpha^{q^i}\beta^{q^j}$, there exists exactly one $\delta \in \{0, 1, \dots, d-1\}$ such that $\alpha^{q^i}\beta^{q^j} \in h_{\delta}$, i.e., the only one satisfying $\delta \equiv j - i \mod{d}$. Since the size of $\{\alpha^{q^i}\beta^{q^j}: 0 \le i \le n-1, 0 \le j \le m-1\}$ is $nm$ and each of these appears in $h_\delta$ for some $\delta \in \{0, 1, \dots, d-1\}$, our claim follows from a counting argument.

     For each $i = 0, 1, \dots, d-1$, $h_i$ consists of $t_i$ copies of some $F$-orbit $f_i$ of degree $\frac{l}{t_i}$. Thus, by \Cref{thm:explicit_epsilon0_factor_for_tame_representation_in_macdonald_parameterization} and \Cref{thm:hasse-davenport}, we have
    $$\epsilon_0(\phi, \psi) = (-1)^{nm}q^{-\frac{nm}{2}}\prod_{i=0}^{d-1} \tau(f_i, \psi_{\frac{l}{t_i}})^{t_i} = (-1)^{nm}q^{-\frac{nm}{2}}\prod_{i=0}^{d-1} \tau(h_i, \psi_l).$$
    Thus, by the definition of $\tau(h_i, \psi_l)$, we have
    $$\epsilon_0(\pi_f \times \pi_g, \psi) = \epsilon_0(\phi, \psi) = (-1)^{nm}q^{-\frac{nm}{2}}\prod_{i=0}^{d-1} \tau(\alpha\beta^{q^i}, \psi_l).$$
\end{proof}

If $\gcd(n, m) = 1$, then $\epsilon_0(\pi_f \times \pi_g, \psi)$ is just a multiple of a Gauss sum $\tau(\alpha\beta, \psi_{nm})$, where $\alpha\beta \in \Gamma_{nm}$ is defined by $(\alpha\beta)(x) = \alpha\left(N_{nm, n}(x)\right)\beta\left(N_{nm, m}(x)\right)$ for $x \in \finiteField_{nm}^\times$. In particular if $m = 1$, we have
\begin{equation*}
    \epsilon_0(\pi_f \times \beta, \psi) = (-1)^n q^{-\frac{n}{2}}\tau(\alpha\beta, \psi_{n}).
\end{equation*}

\subsection{Rankin-Selberg gamma factors}

Let $n > m$ be two positive integers. Let $f$ and $g$ be two $F$-orbits in $\Gamma$ such that $d(f) = n$ and $d(g) = m$. Then $\pi_f$ and $\pi_g$ are two irreducible cuspidal representations of $\G_n(\finiteField)$ and $\G_m(\finiteField)$ respectively. Let $\phi_f = (\rho_f, 0)$ and $\phi_g = (\rho_g, 0)$ be the tamely ramified Weil-Deligne representations of $W$ corresponding to $\pi_f$ and $\pi_g$ respectively. Let $\Pi_f$ be the level zero supercuspidal representation of $\G_n(\localField)$ constructed from the pair $(\pi_f, 1)$, see \cite[Section 2.1]{Ye18}. That is to say,
$$\Pi_f = \mathrm{ind}_{\localField^{\times} \cdot \G_n(\mathfrak{o})}^{\G_n(\localField)} \left(\chi_f \cdot \pi_f\right),$$
where $\pi_f$ is a representation of $\G_n(\mathfrak{o})$ via the natural map $\G_n(\mathfrak{o}) \to \G_n(\finiteField)$, $\chi_f : \localField^{\times} \rightarrow \mathbb{C}^{\times}$ is the character defined on $\mathfrak{o}^{\times}$ via the central character of $\pi_f$, and on $\varpi$ by $\chi_f \left( \varpi \right) = 1$, $\left(\chi_f \cdot \pi_f\right)\left(z \cdot k\right) = \chi_f \left(z\right) \pi_f \left(k_0\right)$ for $z \in \localField^{\times}$ and $k_0 \in \G_n\left(\mathfrak{o}\right)$, and $\mathrm{ind}$ is compact induction. Similarly, let $\Pi_g$ be the level zero supercuspidal representation of $\G_m(\localField)$ constructed from the pair $(\pi_g, 1)$. The choices of $\Pi_f$ and $\Pi_g$ here are not essential, we can choose any level zero supercuspidal representations $\Pi_f$ and $\Pi_g$ constructed from $\pi_f$ and $\pi_g$ respectively.

Let $\gamma(s, \Pi_f \times \Pi_g, \psi)$ and $\epsilon(s, \Pi_f \times \Pi_g, \psi)$ be the Rankin-Selberg factors defined in \cite{JacquetPiShShalika83}, and let $\gamma(\pi_f \times \pi_g, \psi)$ be the Rankin-Selberg gamma factor defined in \cite{Roditty10,Nien14} (see also \cite{Ye18} for a different normalization). $\gamma(\pi_f \times \pi_g, \psi)$ is related to $\gamma(s, \Pi_f \times \Pi_g, \psi)$ by

\begin{theorem}[{\cite[Theorem 3.11]{NienZhang18}, \cite[Theorem 3.1]{Ye18}}]\label{thm:equality_of_rankin_selberg_gamma_factors}
    Let $\omega_g$ be the central character of $\pi_g$. Then
    $$\omega_g(-1)^{n-1} \gamma(s, \Pi_f \times \Pi_g, \psi) = q^{\frac{m(n-m-1)}{2}}   \gamma(\pi_f \times \pi_g, \psi).$$
\end{theorem}

Next, we will show that $\gamma(\pi_f \times \pi_g, \psi)$ is also equal to $\epsilon_0(\pi_f \times \pi_g, \psi)$ up to some constant.

\begin{theorem}\label{thm:equality_of_rankin_selberg_gamma_factor_and_tensor_product_epsilon_factor}
    \begin{equation*}
        \gamma(\pi_f \times \pi_g, \psi) = q^{-\frac{m(n-m-1)}{2}}\omega_g(-1)^{n-1}\epsilon_0(\pi_f \times \pi_g, \psi).
    \end{equation*}
\end{theorem}

\begin{proof}
    Since $n > m$ and $\Pi_f$ and $\Pi_g$ are supercuspidal, $L \left(s, \Pi_f \times \Pi_g \right) = 1$. Similarly, $L\left(1-s, \Pi_f^\vee \times \Pi_g^\vee\right) =1$, where $\Pi_f^\vee$ and $\Pi_g^\vee$ are the contragredient representations of $\Pi_f$, $\Pi_g$ respectively. Thus, we have
    \begin{equation}\label{eqn:analytic_gamma_factor_equals_analytic_epsilon_factor}
        \gamma(s, \Pi_f \times \Pi_g, \psi) = \epsilon(s, \Pi_f \times \Pi_g, \psi).
    \end{equation}
    We have from \cite[Proposition 1 in A.2]{SilbergerZink08} the following commutative diagram:
    \begin{center}
        \begin{tikzcd}
            \Pi_0(\G_n(\localField)) \arrow[r, "LLC_0"] \arrow[d, "p_1" left] & \Phi^t(\G_n) \arrow[d, "p_2" right]\\
            \Pi_0(\G_n(\finiteField)) \arrow[r, "\mathcal{M}"] & \Phi^t_I(\G_n)
        \end{tikzcd},
    \end{center}
    where $\Pi_0(\G_n(\localField))$ is the set of equivalence classes of level zero supercuspidal representations of $\G_n(\localField)$, $\Pi_0(\G_n(\finiteField))$ is the set of equivalence classes of cuspidal representations of $\G_n(\finiteField)$, $LLC_0$ is the restriction of the local Langlands correspondence to level zero supercuspidal representations, $\mathcal{M}$ is Macdonald's correspondence, $p_1$ is the map sending a level zero supercuspidal representation to the cuspidal representation from which it was constructed, and $p_2$ is the canonical projection.

    Therefore, there exist $\phi_f = (\rho_f, 0) \in \Phi^t(\G_n)$ and $\phi_g = (\rho_g, 0) \in \Phi^t(\G_m)$ in the $I$-equivalence classes of tamely ramified Weil-Deligne representations parameterized by $f$ and $g$ respectively that are images of $\Pi_f$ and $\Pi_g$ under the local Langlands correspondence. Since the local Langlands correspondence matches epsilon factors of pairs \cite[Corollary VII.2.17]{HarrisTaylor01}, we get
    \begin{equation}\label{eqn:analytic_epsilon_factor_equals_arithmetic_epsilon_factor}
        \epsilon(s, \Pi_f \times \Pi_g, \psi) = \epsilon(s, \phi_f \otimes \phi_g, \psi, dx).
    \end{equation}
    Note that it is important to normalize $dx$ in \Cref{eqn:analytic_epsilon_factor_equals_arithmetic_epsilon_factor} to be self dual with respect to $\psi$, which is exactly what we do.

    Since $\rho_f$ and $\rho_g$ are tamely ramified, $\rho_f \otimes \rho_g$ is also tamely ramified and
    $$(\rho_f \otimes \rho_g)_I = \sum_{\gamma_1 \in f, \gamma_2 \in g} \gamma_1 \gamma_2.$$
    Since $\rho_f$ and $\rho_g$ are of different degrees, $\gamma_1\gamma_2$ is not fixed by the action of $F$, so $\gamma_1\gamma_2$ in the summand above will never be the trivial character of $\finiteField$. In other words, if $\lambda^\prime \in P_n(\Gamma)$ is the partition-valued function on $\Gamma$ corresponding to $\phi_f \otimes \phi_g$ under \Cref{thm:parameterization_of_tame_representations}, then $\lambda^\prime(1) = ()$.  Therefore, by \Cref{cor:sufficient_condition_for_I_invariant_subspace_to_vanish}, we have $\epsilon(s, \phi_f \otimes \phi_g, \psi, dx) = \epsilon_0(\phi_f \otimes \phi_g, \psi)$. Thus, we conclude the theorem from \Cref{thm:tensor_epsilon_factor_as_a_product_of_gauss_sums}, \Cref{thm:equality_of_rankin_selberg_gamma_factors}, \Cref{eqn:analytic_gamma_factor_equals_analytic_epsilon_factor} and \Cref{eqn:analytic_epsilon_factor_equals_arithmetic_epsilon_factor}.
\end{proof}

As a corollary of \Cref{thm:tensor_epsilon_factor_as_a_product_of_gauss_sums} and \Cref{thm:equality_of_rankin_selberg_gamma_factor_and_tensor_product_epsilon_factor}, we get

\begin{corollary}\label{cor:rankin_selberg_factor_as_a_product_of_gauss_sums}
    Using the notations of \Cref{thm:tensor_epsilon_factor_as_a_product_of_gauss_sums} and \Cref{thm:equality_of_rankin_selberg_gamma_factor_and_tensor_product_epsilon_factor}, we have
    $$\gamma(\pi_f \times \pi_g, \psi) = (-1)^{nm}q^{-nm+\frac{m(m+1)}{2}}\beta(-1)^{n-1}\prod_{i=0}^{d-1} \tau(\alpha\beta^{q^i}, \psi_l).$$
\end{corollary}

\begin{remark}
	After this paper was written, the preprint \cite{Yang2019} of Yang appeared, in which a different proof of \Cref{cor:rankin_selberg_factor_as_a_product_of_gauss_sums} is given.
\end{remark}

\begin{proof}
    It is a direct consequence of \Cref{thm:tensor_epsilon_factor_as_a_product_of_gauss_sums} and \Cref{thm:equality_of_rankin_selberg_gamma_factor_and_tensor_product_epsilon_factor}, with an observation that $\omega_g$, the central character of $\pi_g$, is $\beta|_{\finiteField^\times}$, i.e., the restriction of $\beta$ to $\finiteField^\times$.
\end{proof}

The appearance of $\beta(-1)^{n-1}$ in \Cref{cor:rankin_selberg_factor_as_a_product_of_gauss_sums} is due to the different choices of normalization of the Rankin-Selberg gamma factors in \cite{JacquetPiShShalika83} and \cite{Nien14}. The result when $m=1$ has already been proven in \cite[Theorem 1.1]{Nien17}. The corollary also answers \cite[Conjecture 2.2]{NienZhang18}. Nien and Zhang conjecture that
\begin{equation}\label{eqn:nien_zhang_conjecture}
    \gamma(\pi_f \times \pi_g, \psi) = (-1)^{nm-m+1}q^{-nm+\frac{m(m+1)}{2}}\alpha(-1)^{m-1}\beta(-1)^{n-1}\tau(\alpha\beta, \psi_{nm}),
\end{equation}
which in general is not correct. If $n$ and $m$ are coprime, then their conjecture is correct, possibly up to a sign. However, if they are not coprime, one can find examples in which Nien-Zhang's conjecture does not compute gamma factors correctly, as illustrated in the following example.

\textbf{Example.} Let $\finiteField = \mathbb{F}_3$, the field of $3$ elements. Let $\psi: \finiteField \to \C^\times$ be defined by $\psi(x) = e^{\frac{2\pi i x}{3}}$. Let $n = 4$ and $m = 2$. $l = \mathrm{lcm}(n, m) = 4$ and $d = \gcd(n, m) = 2$. Let $\xi$ be a root of $x^4+2x^3+2=0$ and $\zeta = \xi^{10}$ be a root of $x^2+2x+2=0$. Then $\finiteField_4 = \finiteField(\xi)$ and $\finiteField_2 = \finiteField(\zeta)$, and $\xi$, $\zeta$ are generators of the cyclic groups $\finiteField_4^{\times}$ and $\finiteField_2^{\times}$ respectively. Let $\alpha$ be the multiplicative character of $\finiteField_4$ defined by $\alpha(\xi)=e^{\frac{33\pi i}{20}}$ and $\beta$ be the multiplicative character of $\finiteField_2$ defined by $\beta(\zeta)=e^{\frac{\pi i}{4}}$. Let $f$ be the $F$-orbit of $\alpha$, then $f = \{\alpha, \alpha^3, \alpha^9, \alpha^{27}\}$ with evaluations at $\xi$ being
$$\{e^{\frac{33\pi i}{20}}, e^{\frac{19\pi i}{20}}, e^{\frac{17\pi i}{20}}, e^{\frac{11\pi i}{20}}\}.$$
Similarly, if $g = \{\beta, \beta^3\}$ is the $F$-orbit of $\beta$ with evaluations at $\zeta$ being $\{e^{\frac{\pi i}{4}}, e^{\frac{3\pi i}{4}}\}$. Then $\pi_f$ is an irreducible cuspidal representation of $\G_4(k)$ and $\pi_g$ is an irreducible cuspidal representation of $\G_2(k)$.

Gauss sums can be computed easily with the help of SageMath. We computed the right hand side of the equation in \Cref{cor:rankin_selberg_factor_as_a_product_of_gauss_sums}:
$$(-1)^{nm}q^{-nm+\frac{m(m+1)}{2}}\beta(-1)^{n-1}\prod_{i=0}^{d-1} \tau(\alpha\beta^{q^i}, \psi_l) = -\frac{2}{9} + \frac{\sqrt{5}}{9}i.$$
The right hand side of \Cref{eqn:nien_zhang_conjecture} is
$$(-1)^{nm-m+1}q^{-nm+\frac{m(m+1)}{2}}\alpha(-1)^{m-1}\beta(-1)^{n-1}\tau(\alpha\beta, \psi_{nm}) = -\frac{1}{27} - \frac{4\sqrt{5}}{27}i.$$

To compute $\gamma(\pi_f \times \pi_g, \psi)$, we use the formula involving Bessel functions, see \cite[Proposition 2.8]{Nien17}:
\begin{equation}\label{eqn:gamma_factor_in_terms_of_bessel_functions}
    \gamma(\pi_f \times \pi_g, \psi) = \sum_{h \in U_m \backslash \G_m(\finiteField)}
    B_{f, \psi} \begin{pmatrix} 0 & I_{n-m} \\ h & 0 \end{pmatrix}
    B_{g, \psi^{-1}}(h),
\end{equation}
where $U_m \subset \G_m(\finiteField)$ is the standard maximal unipotent subgroup, $B_{f, \psi}$ is the normalized Bessel function of $\pi_f$ with respect to $\psi$ and $B_{g, \psi^{-1}}$ is the normalized Bessel function of $\pi_g$ with respect to $\psi^{-1}$. To be more explicit, the algorithm runs by the following steps. First, we compute the trace of $\pi_f$ and $\pi_g$ using \cite[(6.1)]{Gelfand70}. Secondly, we compute the normalized Bessel functions using \cite[Proposition 4.5]{Gelfand70}. Finally, we compute the gamma factor using \Cref{eqn:gamma_factor_in_terms_of_bessel_functions}.
We implement the algorithm using SageMath and get
$$\gamma(\pi_f \times \pi_g, \psi) = -\frac{2}{9} + \frac{\sqrt{5}}{9}i.$$

\section{Exterior square epsilon factors}\label{sec:exterior_square_epsilon_factors}

In this section, we perform a similar analysis for $\epsilon_0(\pi, \wedge^2, \psi)$ as in \Cref{sec:tensor_epsilon_factors}. Since some techniques have already been presented in the previous section, we will present only the core of the proofs of some of the theorems.

Let $\lambda \in P_n(\Gamma)$, and $\pi_\lambda$ be the corresponding irreducible representation of $\G_n(\finiteField)$. Suppose that the support of $\lambda$ is $\{f_1, f_2, \dots, f_t\}$, where the $f_i$'s are $F$-orbits in $\Gamma$, i.e., $\lambda(\gamma)$ is a non-empty partition if and only if $\gamma \in f_i$ for some $i$. $\epsilon_0(\pi_\lambda, \wedge^2, \psi)$ satisfies the following multiplicativity theorem.

\begin{theorem}\label{thm:multiplicativity_of_exterior_square_epsilon_factor}
    Denote by $\pi_{f_i}$ the (equivalence class of the) irreducible cuspidal representation corresponding to the $F$-orbit $f_i$ as before. Let $n_i = |\lambda(f_i)|$. Then
    \begin{equation*}
        \epsilon_0(\pi_\lambda, \wedge^2, \psi) 
        = \prod_{1 \le i < j \le t} \epsilon_0(\pi_{f_i} \times \pi_{f_j}, \psi)^{n_in_j} \cdot
          \prod_{i=1}^t \epsilon_0(\pi_{f_i} \times \pi_{f_i}, \psi)^{{n_i \choose 2}} \cdot
          \prod_{i=1}^t \epsilon_0(\pi_{f_i}, \wedge^2, \psi)^{n_i}.
    \end{equation*}
\end{theorem}

\begin{proof}
    Similar to the proof of \Cref{thm:multiplicativity_of_tensor_epsilon_factor}. By definitions,
    $$\epsilon_0(\pi_\lambda, \wedge^2, \psi) = \epsilon_0(\wedge^2(\phi_\lambda), \psi) = \epsilon_0(\wedge^2(\rho_\lambda), \psi).$$
    Since $\rho_\lambda$ is tamely ramified, so is $\wedge^2(\rho)$. Thus, by \Cref{cor:epsilon0_factors_are_defined_on_I_equivence_classes},
    $$\epsilon_0(\wedge^2(\rho_\lambda), \psi) = \epsilon_0(\wedge^2(\rho_\lambda)_I, \psi) = \epsilon_0\left(\wedge^2\left((\rho_\lambda)_I\right), \psi\right).$$
    By \Cref{thm:explicit_epsilon0_factor_for_tame_representation_in_macdonald_parameterization}, we have
    $$(\rho_\lambda)_I = \sum_{f \in F\backslash \Gamma}|\lambda(f)|\sum_{\gamma \in f} \gamma,$$
    and for any $F$-orbit $f$,
    $$(\rho_f)_I = \sum_{\gamma \in f} \gamma.$$
    Therefore
    $$(\rho_\lambda)_I = \sum_{f} |\lambda(f)| (\rho_f)_I.$$
    Now the theorem follows by repeatedly using
    $$\epsilon_0(\wedge^2(\rho_1 + \rho_2), \psi) = \epsilon_0(\rho_1 \otimes \rho_2, \psi)\epsilon_0(\wedge^2(\rho_1), \psi)\epsilon_0(\wedge^2(\rho_2), \psi),$$
    and \Cref{defn:epsilon}.
\end{proof}

By \Cref{thm:tensor_epsilon_factor_as_a_product_of_gauss_sums}, we know how to compute the tensor product epsilon factors in terms of Gauss sums. Thus in order to compute $\epsilon_0(\pi_\lambda, \wedge^2, \psi)$, we just need to compute $\epsilon_0(\pi_f, \wedge^2, \psi)$ in light of \Cref{thm:multiplicativity_of_exterior_square_epsilon_factor}.

\begin{theorem}\label{thm:exterior_square_epsilon_factor_as_product_of_gauss_sums}
    Let $f = \{\alpha^{q^i}: i = 0, 1, \dots, n-1\}$ be an $F$-orbit of degree $n$ for some $\alpha \in \Gamma_n$. Let $m = \left\lfloor \frac{n}{2} \right\rfloor$ be the biggest integer smaller or equal to $\frac{n}{2}$. Then
    \begin{equation*}
        \epsilon_0(\pi_f, \wedge^2, \psi) = (-1)^{{n \choose 2}}q^{-\frac{1}{2}{n \choose 2}}\tau(\alpha^{1+q^m}, \psi_d)\prod_{i=1}^{m-1} \tau(\alpha^{1+q^i}, \psi_n),
    \end{equation*}
    where $d = n$ if $n$ is odd and $d = m$ if $n$ is even.
\end{theorem}

\begin{proof}
    Similar to the proof of \Cref{thm:tensor_epsilon_factor_as_a_product_of_gauss_sums}. The key is that
    $$\left(\wedge^2(\rho_f)\right)_I = \sum_{1 \le i < j \le n} \alpha^{q^i+q^j} = \sum_{j = 1}^{m}\sum_{\beta \in h_j} \beta,$$
    where $h_j$ is the $F$-set generated by $\alpha^{1+q^{j}}$ with respect to $\Gamma_n$ if $1 \le j < \frac{n}{2}$, and if $n = 2m$ is even, then $h_m$ is the $F$-set generated by $\alpha^{1 + q^m}$ with respect to $\Gamma_m$ (Note that when $n=2m$ is even, $\alpha^{q^{2m}} = \alpha$ and
    $(\alpha^{1+q^m})^{q^m} = \alpha^{q^m + q^{2m}} = \alpha^{q^m+1} = \alpha^{1+q^m}$
    so $\alpha^{1+q^m} \in \Gamma_m$). It can be checked that these $h_j$'s, for $1 \le j \le m$, form a partition of $\{\alpha^{q^i+q^j}: 1 \le i < j \le n\}$.
\end{proof}

Let $f = \{\alpha^{q^i}: i = 0, 1, \dots, n-1\}$ be an $F$-orbit in $\Gamma_n$ for some $\alpha \in \Gamma_n$. Assume that $\alpha$ is not trivial when restricted to $\finiteField_{m}^\times$ if $n=2m$, or equivalently, $\alpha^{1+q^m}$ is not the trivial character. This assumption assures that the cuspidal representation $\pi_f$ does not have non-trivial Shalika vectors \cite[Section 2.3.3]{YeZeligher18}. The authors define the exterior square gamma factor $\gamma(\pi_f, \wedge^2, \psi)$ in \cite{YeZeligher18}, and show that
\begin{equation}\label{eqn:equality_of_exterior_square_gamma_factor_and_jacquet_shalika_gamma_factor}
    \gamma(\pi_f, \wedge^2, \psi) = \gamma_{JS}(s, \Pi_f, \wedge^2, \psi),
\end{equation}
where, as before, $\Pi_f$ is a level zero supercuspidal representation constructed from $\pi_f$, and $\gamma_{JS}(s, \Pi_f, \wedge^2, \psi)$ is the Jacquet-Shalika gamma factor from \cite{JacquetShalika90,Matringe14,CogdellMatringe15}. It is natural to ask what is the relation between $\epsilon_0(\pi_f, \wedge^2, \psi)$ and $\gamma(\pi_f, \wedge^2, \psi)$, which might enable us to express $\gamma(\pi_f, \wedge^2, \psi)$ as a product of Gauss sums. Unfortunately, we don't have the matching between $\epsilon_{JS}(s, \Pi_f, \wedge^2, \psi)$ and $\epsilon(s, \wedge^2(\phi_f), \psi)$ under the local Langlands correspondence. Thus, the arguments in \Cref{thm:equality_of_rankin_selberg_gamma_factor_and_tensor_product_epsilon_factor} can not go through. We want to mention that $\epsilon_{LS}(s, \Pi_f, \wedge^2, \psi)$, the exterior square epsilon factor coming from Langlands-Shahidi method, does match $\epsilon(s, \wedge^2(\phi_f), \psi)$. This is the work of Cogdell, Shahidi and Tsai \cite{CogdellShahidiTsai14}. But we don't know whether or not the two exterior square epsilon factors coming from the Jacquet-Shalika integral representation and the Langlands-Shahidi method are the same.

Nevertheless, since $\gamma(\pi_f, \wedge^2, \psi)$ and $\epsilon_0(\pi, \wedge^2, \psi)$ are constants, there must exist a constant $c_f$ depending on $f$ such that
\begin{equation*}
    \begin{split}
        \gamma(\pi_f, \wedge^2, \psi)
        &= c_f \, \epsilon_0(\pi_f, \wedge^2, \psi) \\
        &= c_f \cdot (-1)^{{n \choose 2}}q^{-\frac{1}{2}{n \choose 2}}\tau(\alpha^{1+q^m}, \psi_d)\prod_{i=1}^{m-1} \tau(\alpha^{1+q^i}, \psi_n).
    \end{split}
\end{equation*}
We know from \cite{YeZeligher18} that $|\gamma(\pi_f, \wedge^2, \psi)|=1$ and that the absolute value of a Gauss sum $\tau(\alpha, \psi_n)$ is $q^{\frac{n}{2}}$ if $\alpha$ is not trivial. Therefore, by taking absolute values on both sides of the equation above, we get $|c_f| = 1$. We conjecture that $c_f = 1$. For $n = 2$, $\gamma(\pi_f, \wedge^2, \psi)$ degenerates into a Godement-Jacquet gamma factor of the central character $\omega_f$ of $\pi_f$, thus $c_f = 1$ is a consequence of \Cref{cor:rankin_selberg_factor_as_a_product_of_gauss_sums}. For $n = 3$, $\gamma(\pi_f, \wedge^2, \psi)$ has a nice expression in terms of the Bessel function of $\pi_f$ \cite{YeZeligher18}, which can be manipulated into the $3 \times 1$ Rankin-Selberg gamma factor $\overline{\gamma(\pi_f \times \omega_f^{-1}, \psi)}$. Thus $c_f = 1$ is again a consequence of \Cref{cor:rankin_selberg_factor_as_a_product_of_gauss_sums}. We don't know how to prove $c_f = 1$ for $n \ge 4$. Let $\Pi_f$ be any level zero representation constructed from $\pi_f$. We know that $\epsilon_{JS}(s, \pi_f, \wedge^2, \psi) = \gamma(\pi_f, \wedge^2, \psi)$ \cite{YeZeligher18}. On the other hand, by \Cref{cor:sufficient_condition_for_I_invariant_subspace_to_vanish} and \cite{CogdellShahidiTsai14}, $\epsilon_{LS}(s, \Pi_f, \wedge^2, \psi) = \epsilon_0(\pi_f, \wedge^2, \psi)$. Therefore, $c_f = 1$ is equivalent to $\epsilon_{JS}(s, \Pi_f, \wedge^2, \psi) = \epsilon_{LS}(s, \Pi_f, \wedge^2, \psi)$.

The above discussion requires $\alpha|_{\finiteField_m^\times} \ne 1$ if $n = 2m$, because $\gamma(\pi_f, \wedge^2, \psi)$ is not defined in this special case. Interestingly, if we assume $n = 2m$ and $\alpha|_{\finiteField_m^\times} = 1$, we can show by computations that $\epsilon_{JS}(s, \Pi_f, \wedge^2, \psi) = \epsilon_{LS}(s, \Pi_f, \wedge^2, \psi)$ for any level zero supercuspidal representation $\Pi_f$ constructed from $\pi_f$.

Let $n = 2m$ and $\alpha|_{\finiteField_m^\times} = 1$. We have computed in \cite{YeZeligher18} that
$$\epsilon_{JS}(s, \Pi_f, \wedge^2, \psi) = \omega_{\Pi_f}(\varpi)^{-1}q^{m\left(s - \frac{1}{2}\right)s},$$
where $\omega_{\Pi_f}$ is the central character of $\Pi_f$. Let $\phi_f$ be the image of $\Pi_f$ under the local Langlands correspondence. Let $W = \wedge^2(\phi_f)$. Then from \Cref{eqn:arithmetic_epsilon_of_tame_representation},
$$\epsilon(s, \wedge^2(\phi_f), \psi) = \epsilon(W, \psi)q^{\left(\dim W^I\right)s} = \epsilon_0(W, \psi)\det(-F, W^I)^{-1}q^{\left(\dim W^I\right)s}.$$
We computed that $\epsilon_0(W, \psi) = -q^{-\frac{m}{2}}$ using \Cref{thm:exterior_square_epsilon_factor_as_product_of_gauss_sums} and computations of the Gauss sums, $\det(-F, W^I) = -\omega_{\Pi_f}(\varpi)$ using the local Langlands correspondence, and $\dim W^I = m$ by counting the multiplicity of the trivial character in $W_I = \sum_{1 \le i < j \le n}\alpha^{q^i+q^j}$, where $W_I$ is the restriction of $W$ to $I$. See also \cite[Section 4.4]{yephd19} for details. Therefore,
$$\epsilon(s, \wedge^2(\phi_f), \psi) = \omega_{\Pi_f}(\varpi)^{-1}q^{m\left(s - \frac{1}{2}\right)s} = \epsilon_{JS}(s, \Pi_f, \wedge^2, \psi).$$
We know $\epsilon_{LS}(s, \Pi_f, \wedge^2, \psi) = \epsilon(s, \wedge^2(\phi_f), \psi)$ from \cite{CogdellShahidiTsai14}, so
$$\epsilon_{JS}(s, \Pi_f, \wedge^2, \psi) = \epsilon_{LS}(s, \Pi_f, \wedge^2, \psi).$$

\bibliographystyle{abbrv}
\bibliography{references}
\end{document}